\RequirePackage{fix-cm}
\documentclass[smallcondensed,envcountsect]{svjour3}       
\smartqed  

\usepackage{amsmath,amsxtra, amssymb, latexsym, amscd,cite,comment}
\usepackage{graphicx,color}
\usepackage[active]{srcltx}
\usepackage[utf8]{inputenc}
\usepackage[mathscr]{euscript}
\usepackage{mathrsfs}
\usepackage[english]{babel}
\usepackage{pgf,tikz}
\usetikzlibrary{arrows}
\usepackage{latexsym,amsfonts,verbatim,cite,marvosym}
\usepackage{hyperref}
\begin{document}

\title{On second-order sufficient optimality conditions for $C^1$ vector optimization problems}
\titlerunning{On second-order sufficient optimality conditions for $C^1$ VPs}
\author{Nguyen Van Tuyen$^1$  \and
        Jen-Chih Yao$^2$ \and Ching-Feng Wen$^3$ \and Yi-Bin Xiao$^4$ 
}
\authorrunning{N.V. Tuyen, J.-C. Yao, C.-F. Wen, and Y.-B. Xiao}
	\institute{
		\Letter\ \ Ching-Feng Wen
		 \\
		\text{\ \ \ \ } {cfwen@kmu.edu.tw}
		\\
		\\
		\text{\ \ \ \ } Nguyen Van Tuyen
		\\
		\text{\ \ \ \ } {tuyensp2@yahoo.com; nguyenvantuyen83@hpu2.edu.vn}
		\\
		\\
		\text{\ \ \ \ } Jen-Chih Yao
		\\
		\text{\ \ \ \ } {yaojc@mail.cmu.edu.tw}
		\\
		\\
		\text{\ \ \ \ } Yi-Bin Xiao
		\\
		\text{\ \ \ \ } {xiaoyb9999@hotmail.com}
		\\
		\\
		\at$^1$\text{\ \ } School of Mathematical Sciences, University of Electronic Science and Technology of China, Chengdu, P.R. China; Department of Mathematics, Hanoi Pedagogical University 2, Xuan Hoa, Phuc Yen, Vinh Phuc, Vietnam
		\and
		\at$^2$\text{\ \ } Research Center for Interneural Computing,
		China Medical University Hospital China Medical University, Taichung, 40402, Taiwan 
		\and
		\at$^3$\text{\ \ } Center for Fundamental Science; and Research Center for Nonlinear Analysis and Optimization, Kaohsiung Medical University; Department of Medical Research, Kaohsiung Medical University Hospital, Kaohsiung, 80708, Taiwan
		\and
		\at$^4$\text{\ \ } School of Mathematical Sciences, University of Electronic Science and Technology of China, Chengdu, P.R. China	}

\date{Received: date / Accepted: date}

\maketitle

\begin{abstract}
In this paper, we present some second-order sufficient conditions in terms of the Demyanov--Pevnyi's second-order directional derivatives for  efficiency of $C^1$ vector optimization problems with constraints. Our results improve  and generalize conditions obtained by various authors in recent papers.

\end{abstract}

\section{Introduction and Preliminaries}\label{Introduction}
The study of optimality conditions is one of the most important issue in optimization theory. It is well-known that  the first-order optimality conditions are usually not sufficient for optimality except for convex optimization problems. The second-order optimality conditions not only complement first-order ones in eliminating non-optimal solutions, but they also  give us criteria in recognizing the optimality at a given feasible solution. For $C^2$ (i.e. twice continuously differentiable) constrained optimization problems,  the positive definiteness of the Hessian of the associated Lagrangian function on the null-space of the gradient mappings at a stationary point of the active constraints is a sufficient condition for the optimality at this point; see  \cite{Ben-Tal80,Bonnans2000}. For non-$C^2$-smooth problems, to obtain the second-order optimality conditions, many different kinds of generalized second-order directional derivatives have been proposed; see, for example, \cite{Cominetti,Demyanov1974,Ginchev05,Ginchev10,Ginchev11,Huy16,Huy17,Khanh07,Khanh08,Luu17,Mifflin04,Tuyen18,Xiao18}.   One of them is the well-known Demyanov--Pevnyi  second-order directional derivative; see \cite{Demyanov1974}. The second-order directional derivative defined by Demyanov and Pevnyi  was recognized as an effective tool in  studying second-order optimality conditions of nonsmooth optimization problems; see, for example,  \cite{Bel-Tal82,Bel-Tal85,Ginchev08,Jimenez03,Santos2013}.

Assume that $\phi\colon \mathbb{R}^n \to \mathbb{R}$ is a differentiable function at $x\in X$, where $X$ is a nonempty and open subset of $\mathbb{R}^n$. The {\em second-order directional derivative} (in the sense of Demyanov--Pevnyi) of $\phi$ at $x$ in the direction $d\in\mathbb{R}^n$ is defined by
\begin{equation*}
\phi^{\prime\prime}(x; d):=\lim\limits_{t\downarrow 0} \frac{2}{t^2} \Big[\phi (x+td)-\phi (x)-t\langle\nabla\phi(x), d\rangle\Big].
\end{equation*}
If $\phi^{\prime\prime}(x; d)$ exists and it is finite, then $\phi$ is called second-order directionally differentiable at $x$ in the direction $d$. In \cite{Bel-Tal82,Bel-Tal85}, Ben-Tal and Zowe showed that  the second-order directional derivative (in the sense of Demyanov--Pevnyi) exists for a general class of nonsmooth functions arising in applications, for example, the discrete $l_1$ function, the discrete max function, the exact penalty function, and the exterior penalty function. Furthermore, the authors also gave  explicit formulae to calculate the second-order directional derivatives of these functions; see \cite[Section 3]{Bel-Tal85}.

In \cite{Bel-Tal85}, Ben-Tal and Zowe established some second-order sufficient conditions in terms of the Demyanov--Pevnyi's second-order directional derivative for strict local minimizers of  unconstrained scalar optimization problems with $C^{1,1}$ (i.e. continuously differentiable with locally Lipschitz gradients) data. Thereafter, Ginchev and Ivanov \cite[Theorem 9]{Ginchev08} extended these results to scalar constrained optimization problems. Moreover, by using suitable generalized convex assumptions, the authors obtained some second-order sufficient conditions for a point to be a global minimizer. Recently, by using the second-order directional derivative in the sense of Hadamard,  Jim\'enez and Novo \cite{Jimenez03} obtained some sufficient conditions for strict local efficient solution of order $2$ of vector optimization problems with constraints. As shown in \cite[Section 2]{Jimenez03}, the second-order Hadamard directional differentiability implies the second-order directional differentiability in the sense of Demyanov--Pevnyi, but not vice versa.

Motivated by the works reported in \cite{Bel-Tal85,Jimenez03,Ginchev08}, in this paper, we establish some second-order sufficient optimality conditions in terms of the Demyanov--Pevnyi's second-order directional derivatives for efficiency of the following constrained vector optimization  problem
\begin{align*}
& \text{min}_{\,\mathbb{R}^p_+}\, f(x):=(f_1(x), \ldots, f_p(x))\label{problem} \tag{VP}
\\
&\text{s.t.}\ \ x\in \mathcal{F}:=\{x\in X\,:\, g_1(x)\leqq 0, \ldots, g_m(x)\leqq 0\},
\end{align*}
where  $\mathbb{R}^p_+:=\{(y_1, \ldots, y_p)\in\mathbb{R}^p\;:\; y_j\geqq 0, j=1, \ldots, p\}$ is  the nonnegative orthant of $\mathbb{R}^p$, $X$ is a nonempty open subset of $\mathbb{R}^n$,  $f_j$, $j\in J:=\{1, \ldots, p\}$, and $g_i$, $i\in I:=\{1, \ldots, m\}$,  are $C^{1}$ (i.e., continuously differentiable) real-valued functions defined on $X$.  The obtained results improve  the corresponding results of Ginchev and Ivanov \cite[Theorems 1--4]{Ginchev08}, of Jim\'enez and Novo  \cite[Theorem 5.9]{Jimenez03},  and modify  an incorrect result  in \cite[Theorem 5]{Santos2013}.

The organization of the paper is as follows. In the rest of this section, we recall some basic definitions and facts that we need later on. Section \ref{Sufficient-local} is devoted to investigate second-order sufficient conditions of Fritz-John type for a strict local efficient solution  of order $2$ of \eqref{problem}. In Section \ref{Sufficient-global}, we establish some second-order sufficient conditions of Karush--Kuhn--Tucker type and Fritz-John one for global efficiency of \eqref{problem} under suitable generalized convex assumptions.

In the sequel, we use the following notation and terminology. Fix $n \in {\mathbb{N}}:=\{1, 2, \ldots\}$  and abbreviate $(x_1, x_2, \ldots, x_n)$ by $x.$  The space $\mathbb{R}^n$ is equipped with the usual scalar product $\langle \cdot, \cdot \rangle$ and the corresponding Euclidean norm $\| \cdot\|$. The unit sphere in $\mathbb{R}^n$ is denoted by $\mathbb{S}^n$. We denote by $B(x, \delta)$ the open ball centered at $x$ and radius $\delta$.

Let $\Omega$ be a nonempty and closed subset in $\mathbb{R}^n$ and $\bar x\in
\Omega$. The {\em tangent cone} to $\Omega$ at $\bar x$ is defined by
\begin{equation*}
T(\Omega; \bar x):=\{h\in \mathbb{R}^n\;:\; \exists t_k\to 0^+, \exists h^k\to h, \bar x+t_kh^k\in \Omega, \ \ \forall k\in\mathbb{N}\}.
\end{equation*}
It is well-known that for each $x\in \mathbb{S}^n$, we have
$$T (\mathbb{S}^n; x)=x^\bot,$$
where $x^\bot:=\{u\in \mathbb{R}^n\;:\; \langle x, u\rangle=0\}$.

For $a, b\in\mathbb{R}^p$, by $a\leqq b$, we  mean  $a_j\leqq b_j$ for all $j\in J$; by $a\leq b$, we mean $a\leqq b$ and $a\neq b$; and by $a<b$, we mean $a_j<b_j$ for all $j\in J$.
\begin{definition}[see \cite{Ehrgott2006,Jimenez02}]{\rm Let $\bar x\in \mathcal{F}$. We say that:
\begin{enumerate}[(i)]
	\item  $\bar x$ is a  {\em  global weak efficient solution} (resp., {\em   global efficient solution}, {\em strict global   efficient solution}) of problem \eqref{problem}  if there is no $x\in \mathcal{F}$ satisfying $f(x)<f(\bar x)$ (resp., $f(x)\leq f(\bar x)$, $f(x)\leqq f(\bar x)$ with $x\neq \bar x$).
	
	\item $\bar x$ is a  {\em  strict global  efficient solution of order $2$}   of problem \eqref{problem}  if  there exists a constant $\alpha>0$ such that
	\begin{equation*}
	\left(f(x)+\mathbb{R}^p_+\right)\cap B\left(f(\bar x), \alpha\|x-\bar x\|^2\right)=\emptyset, \ \ \forall x\in \mathcal{F}\setminus\{0\}.
	\end{equation*}
	
	\item    $\bar x$ is a {\em local  weak  efficient solution} (resp., {\em local efficient solution}, {\em strict local   efficient solution}, {\em  strict local  efficient solution of order $2$}) of problem \eqref{problem} if it is a global weak  efficient solution (resp., global efficient solution, strict global     efficient solution,  strict global  efficient solution of order $2$) of the considered problem with the constraint set $U\cap \mathcal{F}$, where $U$ is some neighborhood of $\bar x$.
\end{enumerate}
		}
\end{definition}
Fix $\bar x\in \mathcal{F}$, the {\em active index set} at $\bar x$ is defined by $I(\bar x):=\{i\in I\,:\,g_i(\bar x)=0\}.$
For each $d\in\mathbb{R}^n$, put
\begin{align*}
J(\bar x; d)&:=\{j\in J\,:\,\langle\nabla f_j(\bar x), d\rangle=0\},
\\
I(\bar x; d)&:=\{i\in I(\bar x)\,:\, \langle \nabla g_i(\bar x), d\rangle=0\}
\\
\mathcal{C}(f; \bar x)&:=\{d\in\mathbb{R}^n\,:\, \langle\nabla f_j(\bar x), d\rangle\leqq 0,\ \ j\in J\}.
\end{align*}
We say that $d$ is a {\em critical direction} of problem \eqref{problem}  at $\bar x$   if
\[
\begin{cases}
\langle \nabla f_j(\bar x), d\rangle&\leqq 0, \ \ \ \forall j\in J,
\\
\langle \nabla g_i(\bar x), d\rangle&\leqq 0, \ \ \ \forall i\in I(\bar x).
\end{cases}
\]
The set of all critical direction of problem \eqref{problem} at $\bar x$ is denoted by $\mathcal{C}(\bar x)$. For each $d\in \mathcal{C}(\bar x)$, put
\begin{equation*}
\mathcal{C}(\bar x; d):=\{w\in\mathbb{R}^n\,:\,\langle\nabla g_i(\bar x), w\rangle\leqq 0,\ \  i\in I(\bar x; d)\}.
\end{equation*}

The following lemmas will be needed in the sequel.
\begin{lemma}[{see \cite[Lemma 3]{Santos2013}}]\label{Taylor-expansion} Let $\phi\colon X\subset \mathbb{R}^n\to \mathbb{R}$ be a differentiable function, where $X$ is a nonempty and open set and suppose that $\phi$ is second-order directionally differentiable at $\bar x\in X$ in the direction $d\in\mathbb{R}^n$.  Then, for $t>0$ small enough, it holds
\begin{equation*}
\phi(\bar x+td)-\phi (\bar x) = t\langle\nabla \phi (\bar x), d\rangle+\frac{1}{2}t^2\phi^{\prime\prime}(\bar x; d) + o(t^2).
\end{equation*}
\end{lemma}
\begin{lemma}[{see \cite[Proposition 3.4]{Jimenez02}}]\label{stric-efficient-lemma} Let $\bar x$ be a feasible point of problem \eqref{problem}. Then $\bar x$ is not a strict local  efficient solution of order $2$ of problem \eqref{problem} if and only if there exist sequences $\{x^k\}\subset \mathcal{F}\setminus\{\bar x\}$, $\{a^k\}\subset \mathbb{R}^p_+$, such that $x^k\to \bar x$ and
\begin{equation*}
\lim\limits_{k\to\infty} \dfrac{f(x^k)-f(\bar x)+a^k}{\|x^k-\bar x\|^2}=0.
\end{equation*}
\end{lemma}
\section{Sufficient conditions for a strict local  efficient solution  of order $2$}\label{Sufficient-local}
In this section, we  focus on deriving sufficient optimality conditions of Fritz-John type for a   local strict efficient solution  of order $2$ of \eqref{problem}. The main result is as follows.
\begin{theorem}\label{sufficient conditions-I} Let $\bar x$ be a feasible point of \eqref{problem}.  Suppose that $f_j$, $j\in J$, $g_i$, $i\in I(\bar x)$, are  second-order directionally differentiable at $\bar x$ in every direction $d\in T(\mathcal{F}; \bar x)\cap\mathcal{C}(f;\bar x)$.    If for each  $d\in [ T(\mathcal{F}; \bar x)\cap\mathcal{C}(f;\bar x)]\setminus\{0\}$, the following conditions {\rm (I)} and {\rm(II)} are fulfilled, then $\bar x$ is a strict local  efficient solution of order $2$ of problem \eqref{problem}.
\begin{itemize}
\item [\rm (I).] There is $(\mu, \lambda)\in (\mathbb{R}^p_+\times \mathbb{R}^m_+)\setminus\{(0,0)\}$  satisfying
\begin{align}
&\sum_{j=1}^p \mu_j \nabla f_j(\bar x)+\sum_{i=1}^m \lambda_i\nabla g_i(\bar x)=0,\label{Adjoint equation}
\\
&\sum_{j=1}^p \mu_j f_j^{\prime\prime}(\bar x; d)+\sum_{i=1}^m \lambda_i g_i^{\prime\prime}(\bar x; d)>0,\label{second-order sufficient}
\\
&\lambda_i g_i(\bar{x})=0; \ \ \ i\in I.\label{complementary-condition}
\end{align}
  \item [\rm (II).]
  \begin{equation}\label{add-condition}
  \max_{j\in J(\bar x; d)}  \langle \nabla f_j(\bar x), w\rangle>0,\ \ \ \forall w\in \mathcal{C}(\bar x; d)\cap d^{\bot} \setminus\{0\}.
  \end{equation}
  \end{itemize}
\end{theorem}
\begin{proof} On the contrary, suppose that $\bar x$ is not a strict local  efficient solution of order $2$ of \eqref{problem}. Then, by Lemma \ref{stric-efficient-lemma}, there exist sequences  $\{x^k\}\subset \mathcal{F}\setminus\{\bar x\}$, $\{a^k\}\subset \mathbb{R}^p_+$, such that $x^k\to \bar x$ and
\begin{equation*}
\lim\limits_{k\to\infty} \dfrac{f(x^k)-f(\bar x)+a^k}{\|x^k-\bar x\|^2}=0.
\end{equation*}
Hence, for each $j\in J$ and $k\in\mathbb{N}$, we have
\begin{equation}\label{equ:1}
f_j(x^k)-f_j(\bar x)+a_j^k=o(t_k^2),
\end{equation}
where $t_k:=\|x^k-\bar x\|$.

For each $k\in\mathbb{N}$, put  $d^k:=\frac1{t_k}(x^k-\bar x)$. 	Then, $\|d^k\|=1$ for all $k\in\mathbb{N}$. So, without any loss of generality, we may assume that $\{d^k\}$ converges to some $d\in\mathbb{R}^n$ with $\|d\|=1$. Clearly, $d\in T(\mathcal{F}; \bar x)$.
	
	We claim that $d\in \mathcal{C}(f; \bar x)$. Indeed, for each $j\in J$ and $k\in\mathbb{N}$, we have
\begin{equation*}
 f_j(x^k)-f_j(\bar x)= [f_j(\bar x+t_kd^k)-f_j(\bar x+t_kd)]+[f_j(\bar x+t_kd)-f_j(\bar x)].
\end{equation*}
By the Mean Value Theorem for differentiable functions, there exists $\xi_{j}^k\in (\bar x+t_kd^k, \bar x+t_kd)$ satisfying
	\begin{equation*}
	f_j(\bar x+t_kd^k)-f_j(\bar x+t_kd)= t_k\langle \nabla f_j(\xi_{j}^k), d^k-d\rangle.
	\end{equation*}
By  Lemma \ref{Taylor-expansion}, we have
\begin{equation*}
f_j(\bar x+t_kd)-f_j(\bar x)=t_k\langle \nabla f_j(\bar x), d\rangle+\frac{1}{2}t_k^2 f^{\prime\prime}_j(\bar x; d) +o(t^2_k).
\end{equation*}
Hence, by \eqref{equ:1}, we have
\begin{align*}
t_k\langle \nabla f_j(\xi_{j}^k), d^k-d\rangle+t_k\langle \nabla f_j(\bar x), d\rangle+\frac{1}{2}t_k^2 f^{\prime\prime}_j(\bar x; d)+a_j^k=&f_j(x^k)-f_j(\bar x)
\\
&+o(t^2_k)+a_j^k= o(t^2_k).
\end{align*}
This implies that
\begin{equation}\label{equ:4}
0\geqq -\frac{1}{t_k}a^k_j=\langle \nabla f_j(\xi_{j}^k), d^k-d\rangle+\langle \nabla f_j(\bar x), d\rangle+\frac{1}{2}t_k f^{\prime\prime}_j(\bar x; d)+\frac{1}{t_k}p_j^k(t_k),
\end{equation}
where $p_j^k(t_k)=o(t_k^2)$. Since $\xi_{j}^k\to \bar x$, $d^k\to d$ as $k\to \infty$, and $f_j\in C^1(\mathbb{R}^n)$, letting $k\to \infty$ in \eqref{equ:4}, we obtain
\begin{equation*}
\langle \nabla f_j(\bar x), d\rangle\leqq 0, \ \ j\in J,
\end{equation*}
as required.

By Lemma \ref{Taylor-expansion} and the Mean Value Theorem for differentiable functions, for each $i\in I(\bar x)$ and $k\in\mathbb{N}$, there exists  $\eta_{i}^k\in (\bar x+t_kd^k, \bar x+t_kd)$ satisfying
\begin{equation*}
0\geqq g_i(x^k)=t_k\langle \nabla g_i(\eta_{i}^k), d^k-d\rangle+t_k\langle \nabla g_i(\bar x), d\rangle+\frac{1}{2}t_k^2 g^{\prime\prime}_i(\bar x; d) +q_i^k(t_k),
\end{equation*}
where $q_i^k(t_k)=o(t^2_k)$. Thus,
\begin{equation}\label{equ:5}
\langle \nabla g_i(\eta_{i}^k), d^k-d\rangle+\langle \nabla g_i(\bar x), d\rangle+\frac{1}{2}t_k g^{\prime\prime}_i(\bar x; d) +\frac{1}{t_k}q_i^k(t_k)\leqq 0.
\end{equation}

Let  $(\mu, \lambda)\in \mathbb{R}^p_+\times\mathbb{R}^m_+$  be a nonzero Lagrange multiplier  satisfying conditions   \eqref{Adjoint equation}--\eqref{complementary-condition}. Now, multiplying \eqref{equ:4} by $\mu_j$ and \eqref{equ:5} by $\lambda_i$ and summing the inequations obtained, we obtain
\begin{align}
&\sum_{j=1}^p\mu_j\left[\langle \nabla f_j(\xi_{j}^k), d^k-d\rangle+\langle \nabla f_j(\bar x), d\rangle+\frac{1}{2}t_k f^{\prime\prime}_j(\bar x; d)+\frac{1}{t_k}p^k_j(t_k)\right]\notag
\\
&+\sum_{i\in I(\bar x)} \lambda_i\left[\langle \nabla g_i(\eta_{i}^k), d^k-d\rangle+\langle \nabla g_i(\bar x), d\rangle+\frac{1}{2}t_k g^{\prime\prime}_i(\bar x; d)+\frac{1}{t_k}q_i^k(t_k)\right]\leqq 0.\label{equ:6}
\end{align}
Since \eqref{Adjoint equation} and \eqref{complementary-condition}, we see that   \eqref{equ:6} is equivalent to
\begin{align}
&\sum_{j=1}^p\mu_j\left[\langle \nabla f_j(\xi_{j}^k), d^k-d\rangle+ \frac{1}{2}t_k f^{\prime\prime}_j(\bar x; d)+\frac{1}{t_k}p^k_j(t_k)\right]\notag
\\
&+\sum_{i\in I(\bar x)} \lambda_i\left[\langle \nabla g_i(\eta_{i}^k), d^k-d\rangle +\frac{1}{2}t_k g^{\prime\prime}_i(\bar x; d)+\frac{1}{t_k}q_i^k(t_k)\right]\leqq 0.\label{equ:7}
\end{align}
For each $k\in \mathbb{N}$, put $r_k:=\|d^k-d\|$ and $w^k:=\frac{d^k-d}{r_k}$. By the boundedness of $\{w^k\}$, without any loss of generality, we may assume that $\{w^k\}$ converges to some $w\in\mathbb{R}^n$ with $\|w\|=1$. We now rewrite  \eqref{equ:7} as follows:
\begin{align}\label{equ:8}
\begin{split}
\sum_{j=1}^p\mu_j\bigg[r_k\langle \nabla f_j(\xi_{j}^k), w^k\rangle&+\frac{1}{2}t_k f^{\prime\prime}_j(\bar x; d)+\frac{1}{t_k}p^k_j(t_k)\bigg]
\\
&+\sum_{i=1}^m \lambda_i\left[r_k\langle \nabla g_i(\eta_{i}^k), w^k\rangle+\frac{1}{2}t_k g^{\prime\prime}_i(\bar x; d)+\frac{1}{t_k}q_i^k(t_k)\right]\leqq 0.
\end{split}
\end{align}
By passing to a subsequence if necessary, we may consider three cases of the sequence $\big\{\frac{r_k}{t_k}\big\}$ as follows.
\\
{\bf Case 1.} $\displaystyle\lim_{k\to\infty}\frac{r_k}{t_k}=0$. Dividing the two sides of \eqref{equ:8} by $\frac 12t_k$, gives
\begin{align}\label{equ:9}
\begin{split}
\sum_{j=1}^p\mu_j\bigg[\frac{2r_k}{t_k}\langle \nabla f_j(\xi_{j}^k), w^k\rangle&+ f^{\prime\prime}_j(\bar x; d)+\frac{2}{t^2_k}p^k_j(t_k)\bigg]
\\
&+\sum_{i=1}^m \lambda_i\left[\frac{2r_k}{t_k}\langle \nabla g_i(\eta_{i}^k), w^k\rangle+g^{\prime\prime}_i(\bar x; d)+\frac{2}{t^2_k}q_i^k(t_k)\right]\leqq 0.
\end{split}
\end{align}
Letting $k\to\infty$ in \eqref{equ:9}, we obtain
\begin{equation*}
\sum_{j=1}^p\mu_jf^{\prime\prime}_j(\bar x; d)+\sum_{i=1}^m \lambda_ig^{\prime\prime}_i(\bar x; d)\leqq  0,
\end{equation*}
contrary to \eqref{second-order sufficient}.
\\
{\bf Case 2.} $\displaystyle\lim_{k\to\infty}\frac{r_k}{t_k}=:\rho>0$. Letting $k\to\infty$ in \eqref{equ:9}, one has
\begin{equation*}
\sum_{j=1}^p\mu_j\left[2\rho\langle \nabla f_j(\bar x), w\rangle+f^{\prime\prime}_j(\bar x; d)\right]+\sum_{i=1}^m \lambda_i\left[2\rho\langle \nabla g_i(\bar x), w\rangle+ g^{\prime\prime}_i(\bar x; d)\right]\leqq 0,
\end{equation*}
or, equivalently,
\begin{equation*}
2\rho\left[\left\langle \sum_{j=1}^p \mu_j \nabla f_j(\bar x)+\sum_{i=1}^m \lambda_i\nabla g_i(\bar x), w\right\rangle\right] +\sum_{j=1}^p\mu_jf^{\prime\prime}_j(\bar x; d)+\sum_{i=1}^m \lambda_ig^{\prime\prime}_i(\bar x; d)\leqq  0.
\end{equation*}
By \eqref{Adjoint equation}, we have
\begin{equation*}
\sum_{j=1}^p\mu_jf^{\prime\prime}_j(\bar x; d)+\sum_{i=1}^m \lambda_ig^{\prime\prime}_i(\bar x; d)\leqq  0,
\end{equation*}
again contrary to \eqref{second-order sufficient}.
\\
{\bf Case 3.} $\displaystyle\lim_{k\to\infty}\frac{r_k}{t_k}=+\infty$. This means that  $\displaystyle\lim_{k\to\infty}\frac{t_k}{r_k}=0$. Substituting $d^k-d=r_kw^k$ into \eqref{equ:4} and \eqref{equ:5}, we obtain
\begin{align}
r_k\langle \nabla f_j(\xi_{j}^k), w_k\rangle+ \langle \nabla f_j(\bar x), d\rangle+\frac{1}{2}t_k  f^{\prime\prime}_j(\bar x; d) +\frac{1}{t_k}p^k_j(t_k)&\leqq 0,\label{equ:11}
\\
r_k\langle \nabla g_i(\eta_{i}^k), w_k\rangle+ \langle \nabla g_i(\bar x), d\rangle+\frac{1}{2}t_k  g^{\prime\prime}_i(\bar x; d) +\frac{1}{t_k}q^k_i(t_k)&\leqq 0,\label{equ:12}
\end{align}
for all $j\in J$, $i\in I(\bar x)$, and $k\in\mathbb{N}$.

We claim that $w\in \mathcal{C}(\bar x; d)\cap d^{\bot} \setminus\{0\}$. Indeed, since $d^k=d+r_kw^k\to d$, $w^k\to w$ as $k\to\infty$ and $d^k=d+r_kw^k\in \mathbb{S}^n$ for all $k\in \mathbb{N}$, we have $w\in T(\mathbb{S}^n; d)$. Since $T(\mathbb{S}^n; d)=d^{\bot}$, we have that  $w\in d^{\bot}\setminus\{0\}$. From \eqref{equ:12}, for each $i\in I(\bar x, d)$, one has
\begin{equation}\label{equ:13}
\langle \nabla g_i(\eta_{i}^k), w_k\rangle+\frac{1}{2}\frac{t_k}{r_k} g^{\prime\prime}_i(\bar x; d) +\frac{t_k}{r_k}\frac{q^k_i(t_k)}{t^2_k}\leqq 0.
\end{equation}
Letting $k\to\infty$ in \eqref{equ:13}, we obtain $\langle\nabla g_i(\bar x), w\rangle\leqq 0$ for all $i\in I(\bar x, d)$. Consequently, $w\in \mathcal{C}(\bar x; d)\cap d^{\bot} \setminus\{0\}$. From \eqref{equ:11}, for each $j\in J(\bar x; d)$, one has
\begin{equation}\label{equ:14}
\langle \nabla f_j(\xi_{j}^k), w_k\rangle+\frac{1}{2}\frac{t_k}{r_k}f^{\prime\prime}_j(\bar x; d)+\frac{t_k}{r_k}\frac{p^k_j(t_k)}{t_k^2}\leqq 0.
\end{equation}
Letting $k\to\infty$ in \eqref{equ:14}, we have $\langle \nabla f_j(\bar x), w\rangle\leqq 0$ for all $j\in J(\bar x; d)$. Therefore,
\begin{equation*}
\max_{j\in J(\bar x; d)}  \langle \nabla f_j(\bar x), w\rangle\leqq 0,
\end{equation*}
contrary to \eqref{add-condition}. The proof is complete. $\hfill\Box$
\end{proof}
\begin{remark}{\rm   In \cite{Jimenez03},  Jim\'enez and Novo  obtained some second-order sufficient conditions in terms of the second-order Hadamard directional derivative  for  strict local  efficient solutions of order $2$ of constrained vector optimization problems. Recall that a function $\phi\in C^1(X)$ is called {\em second-order Hadamard directional differentiable} at $\bar x\in X$ in the direction $d\in\mathbb{R}^n$ if there exists
\begin{equation*}
d^2\phi(\bar x; d):=\mathop{\lim\limits_{t\downarrow 0}}
	\limits_{u\to d}\frac{2}{t^2} \left[\phi(\bar x+tu)-\phi(\bar x)-t\langle\nabla\phi(\bar x), u\rangle\right].
\end{equation*}
The function $\phi$ is called second-order Hadamard directional differentiable  at $\bar x$   if $d^2\phi(\bar x; d)$ exists for all $d\in \mathbb{R}^n$. Clearly, if $d^2\phi(\bar x; d)$ exists, then so does $\phi^{\prime\prime}(\bar x; d)$ and they are the same. On the other hand, if $\phi^{\prime\prime}(\bar x; d)$ exists and $\nabla\phi(\cdot)$ is stable at $\bar x$, i.e., there are $L\geqq 0$ and $\delta>0$ such that 
$$\|\nabla\phi(x)-\nabla\phi(\bar x)\|\leqq L\|x-\bar x\|, \ \ \forall x\in B(\bar x, \delta),$$
then $d^2\phi(\bar x; d)$ also exists and  $d^2\phi(\bar x; d)=\phi^{\prime\prime}(\bar x; d)$; see \cite[Proposition 2.4]{Jimenez03}. This fact does not hold if $\nabla\phi(\cdot)$ is not stable at $\bar x$; see Example \ref{example1} below.  Jim\'enez and Novo \cite[Theorem 5.9]{Jimenez03} showed that if $f_j$, $j\in J$, $g_i$, $i\in I$, are second-order Hadamard directional differentiable  at $\bar x$ and   for each  $d\in [ T(\mathcal{F}; \bar x)\cap\mathcal{C}(f;\bar x)]\setminus\{0\}$,  there is $(\mu, \lambda)\in (\mathbb{R}^p_+\times \mathbb{R}^m_+)\setminus\{(0,0)\}$  satisfying conditions \eqref{Adjoint equation}--\eqref{complementary-condition}, then $\bar x$ is a strict local  efficient solution of order $2$ of problem \eqref{problem}.  Consequently, if  $f_j$, $j\in J$, $g_i$, $i\in I(\bar x)$, are of class $C^{1,1}(X)$, we can remove   condition \eqref{add-condition} from Theorem \ref{sufficient conditions-I}.  Recently, Ginchev and Ivanov \cite[Example 4]{Ginchev08} introduced a nice example to show that conditions \eqref{Adjoint equation}--\eqref{complementary-condition} are  not sufficient for a point $\bar x$ to be a strict local  efficient solution of order $2$ of scalar optimization problems with $C^1$ data only.   Therefore condition \eqref{add-condition} cannot be dropped in the formulation of  Theorem \ref{sufficient conditions-I}, if there is not any other additional condition.}
\end{remark}

\begin{example}\label{example1} {\rm  Let $f=(f_1, f_2)\colon \mathbb{R}^2\to\mathbb{R}^2$, $g\colon\mathbb{R}^2\to \mathbb{R}$, and $X$ be defined by
\begin{align*}
f_1(x)&:=
\begin{cases}
x_1^{\frac{7}{3}}\sin \frac{1}{x_1}+x_2& \text{if}\ \ x_1\neq 0,
\\
x_2 & \text{if}\ \ x_1= 0,
\end{cases}
\\
f_2(x)&:=x_1, g(x):=x_1^2-x_2 \ \ \forall x=(x_1, x_2)\in X,
\\
X&:=\mathbb{R}^2.
\end{align*}
Clearly, $f_1\in C^1(\mathbb{R}^2)$, $f_2, g\in C^2(\mathbb{R}^2)$ and  the feasible set of \eqref{problem} is
\begin{equation*}
\mathcal{F}=\{(x_1, x_2)\in\mathbb{R}^2\;:\; x_1^2-x_2\leqq 0\}.
\end{equation*}
By simple calculations, one has
\begin{align*}
\nabla f_1(x)&=
\begin{cases} \left(
\frac{7}{3}x_1^{\frac{4}{3}}\sin \frac{1}{x_1}-x_1^{\frac{1}{3}}\cos \frac{1}{x_1}, 1\right)^T& \text{if}\ \ x_1\neq 0,
\\
(0, 1)^T & \text{if}\ \ x_1= 0,
\end{cases}
\\
\nabla f_2(x)&=(1, 0)^T, \nabla g(x)=(2x_1, -1)^T \ \ \forall x=(x_1, x_2)\in\mathbb{R}^2.
\end{align*}

Since  $\nabla f_1(\bar x)=(0, 1)^T$, $\nabla f_2(\bar x)=(1, 0)^T$ and  $\nabla g(\bar x)=(0, -1)^T$, we have
\begin{align*}
\mathcal{C}(\bar x)=\{(d_1, d_2)\in\mathbb{R}^2\;:\; d_1\leqq 0, d_2=0\}.
\end{align*}		
For each $d=(d_1, d_2)\in \mathcal{C}(\bar x; d)\setminus\{0\}$, we have $d=(d_1, 0)\in \mathcal{C}(\bar x)$ with $d_1<0$, and
\begin{align*}
&J(\bar x; d)=\{1\}, I(\bar x; d)=I(\bar x)=I,
\\
&f_1^{\prime\prime}(\bar{x}; d)=f_2^{\prime\prime}(\bar{x}; d)=0, \ \ \text{and}\ \ g^{\prime\prime}(\bar{x}; d)=2d_1^2.
\end{align*}
Thus we can choose $(\mu_1, \mu_2, \lambda)=(1, 0, 1)$ satisfying all conditions \eqref{Adjoint equation}--\eqref{complementary-condition}. Besides, we see that
\begin{align*}
&\mathcal{C}(\bar x; d)\cap d^\bot=\{(w_1, w_2)\in\mathbb{R}^2\;:\; w_1=0, w_2\geqq 0\}.
\end{align*}
Hence, if $w=(w_1,w_2)\in\mathcal{C}(\bar x; d)\cap d^\bot\setminus\{0\}$, then $w_1=0,\,\,w_2>0$, and this implies
\[\max_{j\in J(\bar x; d)}  \langle \nabla f_j(\bar x), w\rangle=w_2>0,\]
 which says that  condition \eqref{add-condition} is satisfied for all $d\in [ T(\mathcal{F}; \bar x)\cap\mathcal{C}(f;\bar x)]\setminus\{0\}$. By  Theorem \ref{sufficient conditions-I}, $\bar x$ is a strict local  efficient solution of order $2$ of problem \eqref{problem}.

In fact, we can check that  $\nabla f_1(\cdot)$ is not stable at $\bar x$ and $d^2f_1(\bar x; d)$ does not exist for all $d\in \mathcal{C}(\bar x)\setminus\{0\}$. Thus \cite[Theorem 5.9]{Jimenez03} cannot be applied for this example.
}
\end{example}

\section{Sufficient conditions for global efficiency}
\label{Sufficient-global}
In this section, under suitable convex assumptions, we introduce some second-order sufficient conditions of Karush--Kuhn--Tucker type and Fritz-John one for global efficiency of \eqref{problem}. In order to formulate these results, we first recall some concepts of generalized convexity from \cite{Ginchev07,Ginchev08,Hanson81,Mangasarian69}.
\begin{definition}[{see \cite{Mangasarian69}}]{\rm   Let  $\phi\colon X\to \mathbb{R}$ be a real-valued function   and $\bar x\in X$. The function $\phi$ is said to be {\em quasiconvex} at $\bar x$ (with respect to $X$) if the conditions $y\in X$, $\phi (y)\leqq \phi (\bar x)$, $t\in [0, 1]$, $(1-t)\bar x+ty\in X$ imply $\phi (\bar x+  t(y-\bar x))\leqq \phi (\bar x)$. If $\phi$ is quasiconvex at every $x\in X$, then we say that $\phi$ is quasiconvex on $X$.
		}	
\end{definition}

The following result is well-known and it could be found in \cite[Theorem 9.1.4]{Mangasarian69}.
\begin{lemma}\label{lemma-quasiconvex} Let $\phi\colon X\to \mathbb{R}$ be a function defined on   $X$ which is both differentiable and quasiconvex at  $\bar x$. Then the following implication holds:
	\begin{equation}\label{quasiconvex-character}
	\left(y\in X, \phi (y)\leqq \phi (\bar x)\right) \Longrightarrow	\langle\nabla\phi (\bar x), y-\bar x\rangle\leqq 0.
	\end{equation}
\end{lemma}

\begin{definition}[{see \cite{Tuy64}}]{\rm  Suppose that the function $\phi\colon X \to \mathbb{R}$ is  differentiable at $\bar x\in X$. We say that $\phi$ is {\em pseudoconvex} at $\bar x$ if $y\in X$ and $\phi(y)<\phi(\bar x)$ imply $\langle \nabla\phi (\bar x), y-\bar x\rangle<0$.
		
		}	
\end{definition}
 \begin{definition}[{see \cite{Ginchev06}}]{\rm  Let $\phi\colon X\to \mathbb{R}$ be a differentiable function at $\bar x\in X$.    Suppose that $\phi$ is second-order directionally differentiable at $\bar x$ in every direction $y-\bar x$ such that $y\in X$, $\phi(y)<\phi(\bar x)$, $\langle\nabla\phi (\bar x), y-\bar x\rangle=0$. We say that $\phi$ is {\em second-order pseudoconvex} (for short, {\em $2$-pseudoconvex}) at  $\bar x$ if, for all $y\in X$, the following implications hold:
\begin{align*}
&\phi (y)<\phi (\bar x)\ \ \text{implies}\ \ \langle\nabla\phi (\bar x), y-\bar x\rangle\leqq 0;
\\
&\phi (y)<\phi (\bar x)\ \ \text{and}\ \ \langle\nabla\phi (\bar x), y-\bar x\rangle=0 \ \ \text{imply}\ \ \phi^{\prime\prime}(\bar x, y-\bar x)<0.
\end{align*}
}
\end{definition}
\begin{remark}{\rm
	Clearly, if $\phi$ is pseudoconvex at $\bar x$, then it  is also $2$-pseudoconvex at this point. The converse does not hold. For example, let $\phi\colon\mathbb{R}\to \mathbb{R}$ be a function defined by
	\begin{equation*}
	\phi(x):=
	\begin{cases}
	x^2 \ \ &\text{if}\ \ x\geq 0,
	\\
	-x^2\ \ &\text{if} \ \ x<0,
	\end{cases}
	\end{equation*}
	and $\bar x=0$. We see that $\phi (y)<\phi(\bar x)$ if and only if $y<0$. Since $\nabla \phi(\bar x)=0$, $\phi$ is not pseudoconvex at $\bar x$. For each $y<0$, one has
	$$\phi^{\prime\prime}(\bar x; y-\bar x)=\lim\limits_{t\downarrow 0}2\frac{\phi(ty)}{t^2}=-2y^2<0.$$
	This implies that $\phi$ is $2$-pseudoconvex at $\bar x$.
}	
\end{remark}

The following result gives sufficient conditions of Karush--Kuhn--Tucker type for a global weak  efficient solution of \eqref{problem} and generalizes \cite[Theorem 1]{Ginchev08} to the vector optimization case.
\begin{theorem}\label{sufficient conditions-II} Let $\bar x$ be a feasible point of \eqref{problem}. Suppose that $f_j$, $j\in J$, $g_i$, $i\in I(\bar x)$ are  second-order directionally differentiable at $\bar x$ in every critical direction $d\in \mathcal{C}(\bar x)$, $f_j$, $j\in J$, are $2$-pseudoconvex at $\bar x$, $g_i$, $i\in I(\bar x)$ are quasiconvex at $\bar x$. If for each $d\in \mathcal{C}(\bar x)$, there exist $\mu\in\mathbb{R}_+^p\setminus\{0\}$ and $\lambda\in\mathbb{R}_+^m$ such that
\begin{align}
&\sum_{j=1}^p \mu_j \nabla f_j(\bar x)+\sum_{i=1}^m \lambda_i\nabla g_i(\bar x)=0, \label{first-order-condition}
\\
&\sum_{j=1}^p \mu_j f_j^{\prime\prime}(\bar x; d)+\sum_{i=1}^m \lambda_i g_i^{\prime\prime}(\bar x; d)\geqq 0, \label{secon-order-condition}
\\
&\lambda_i g_i(\bar{x})=0, \ \ \ i\in I,\label{complementary-condition-II}
\end{align}	
then $\bar x$ is a global weak efficient solution of \eqref{problem}.
\end{theorem}
\begin{proof} Assume the contrary that there exists $x\in \mathcal{F}$ satisfying $f(x)<f(\bar x)$, i.e., $f_j(x)<f_j(\bar x)$ for all $j\in J$. We claim that $x-\bar x$ is a critical direction at $\bar x$. By the $2$-pseudoconvexity of $f_j$, we have $\langle\nabla f_j(\bar x), x-\bar x\rangle\leqq 0$ for all $j\in J$. From the quasiconvexity of $g_i$ and $g_i(x)\leqq g_i(\bar x)$, $i\in I(\bar x)$, we have $\langle \nabla g_i(\bar x), x-\bar x\rangle\leqq 0$ for all $i\in I(\bar x)$. Thus, $x-\bar x$ is a critical direction at $\bar x$. By the assumptions of the theorem, there exist $\mu\in\mathbb{R}_+^p\setminus\{0\}$ and  $\lambda\in\mathbb{R}_+^m$ satisfying conditions \eqref{first-order-condition}--\eqref{complementary-condition-II}. Clearly, $\lambda_i=0$ when $i\notin I(\bar x)$. Since $x-\bar x\in \mathcal{C}(\bar x)$  and \eqref{first-order-condition}, we have
\begin{align*}
0&=\left\langle \sum_{j=1}^p \mu_j \nabla f_j(\bar x)+\sum_{i=1}^m \lambda_i\nabla g_i(\bar x), x-\bar x\right\rangle
\\
&= \sum_{j=1}^p \mu_j \langle\nabla f_j(\bar x), x-\bar x\rangle+\sum_{i\in I(\bar x)} \lambda_i\langle\nabla g_i(\bar x), x-\bar x\rangle\leqq 0.
\end{align*}
This implies that
\begin{equation}\label{system-critical}
\begin{cases}
&\mu_j \langle\nabla f_j(\bar x), x-\bar x\rangle=0,\ \ j\in J,
\\
&\lambda_i\langle\nabla g_i(\bar x), x-\bar x\rangle=0,\ \ i\in I(\bar x).
\end{cases}
\end{equation}
Denote
\begin{equation}\label{supp}
    \text{supp}\,\mu:=\{j\in J\,\,:\,\, \mu_j>0\}\ \ \text{and} \ \ \text{supp}\,\lambda:=\{i\in I\,\,:\,\, \lambda_i>0\}.
\end{equation}
Clearly, $\text{supp}\,\mu\neq \emptyset$ and $\text{supp}\,\lambda\subset I(\bar x)$.  Since \eqref{system-critical}, we have
\begin{equation*}
\begin{cases}
&\langle\nabla f_j(\bar x), x-\bar x\rangle=0,\ \ j\in \text{supp}\,\mu,
\\
&\langle\nabla g_i(\bar x), x-\bar x\rangle=0,\ \ i\in \text{supp}\,\lambda.
\end{cases}
\end{equation*}
By the $2$-pseudoconvexity of $f_{j}$,  one has $f_{j}^{\prime\prime}(\bar x, x-\bar x)<0$ for all $j\in \text{supp}\,\mu$. Moreover, by the quasiconvexity of $g_i$, we have $g_i(\bar x+t(x-\bar x))\leqq 0$ for all $t\in [0, 1]$ and $i\in I(\bar x)$. It follows that
\begin{align*}
g_i^{\prime\prime}(\bar x;  x-\bar x)&=\lim\limits_{t\downarrow 0}2\dfrac{g_i(\bar x+t(x-\bar x))-g_i(\bar x)-t\langle\nabla g_i(\bar x), x-\bar x\rangle}{t^2}
\\
&=\lim\limits_{t\downarrow 0}2\dfrac{g_i(\bar x+t(x-\bar x))}{t^2}\leqq 0
\end{align*}
for all $i\in \text{supp}\,\lambda$.  Therefore,
\begin{align*}
\sum_{j=1}^p \mu_j f_j^{\prime\prime}(\bar x; x-\bar x)+\sum_{i=1}^m \lambda_i g_i^{\prime\prime}(\bar x;  x-\bar x)&=\sum_{j\in \text{supp}\,\mu} \mu_j f_j^{\prime\prime}(\bar x; x-\bar x)
\\
&+\sum_{i\in \text{supp}\,\lambda} \lambda_i g_i^{\prime\prime}(\bar x;  x-\bar x)
\\
&\leqq \sum_{j\in \text{supp}\,\mu} \mu_j f_j^{\prime\prime}(\bar x; x-\bar x)< 0,
\end{align*}
contrary to \eqref{secon-order-condition}. $\hfill\Box$
\end{proof}
The following example  illustrates  Theorem  \ref{sufficient conditions-II}.
\begin{example}\label{linear-example}
{\rm Consider the following linear vector optimization problem:
	\begin{align}
	& \text{min}_{\,\mathbb{R}^2_+}\, f(x):=(f_1(x), f_2(x)) \tag{LVP}\label{linear-problem-exam}
	\\
	&\text{s.t.}\ \ x\in \mathcal{F}:=\{x\in\mathbb{R}^2\,:\, g(x)\leqq 0\},\notag
	\end{align}
	where $f_1(x_1, x_2):=x_1, f_2(x_1,x_2):=x_2,$ and $g(x_1, x_2):=-x_2.$	Let $\bar x=(0,0)\in \mathcal{F}$. The constraint function $g$ is linear, therefore quasiconvex. An easy computation shows that $f_1$ and $f_2$ are  $2$-pseudoconvex at $\bar x$. Since $\nabla f_1(\bar x)=(1, 0)^T$, $\nabla f_2(\bar x)=(0, 1)^T$ and $\nabla g(\bar x)=(0, -1)^T$, we have
	$$\mathcal{C}(\bar x)=\{(d_1, d_2)\,:\, d_1\leqq 0, d_2=0\}.$$
For each $d\in \mathcal{C}(\bar x)$, we can choose $\mu_1=0, \mu_2=1$ and $\lambda=1$   satisfying conditions \eqref{first-order-condition}--\eqref{complementary-condition-II}.  By  Theorem \ref{sufficient conditions-II}, $\bar x$ is a global weak  efficient solution of \eqref{linear-problem-exam}.
	}
\end{example}

By introducing the concept of strictly $2$-pseudoconvex function, Ginchev and Ivanov \cite[Theorems 3 and 4]{Ginchev08} presented some sufficient optimality conditions for  strict global solutions of scalar optimization problems.  We recall here the definition of strictly $2$-pseudoconvex functions.
\begin{definition}{\rm Suppose that  $\phi\colon X \to \mathbb{R}$ is a differentiable function at $\bar x\in X$ and second-order directionally differentiable at this point in every direction $y-\bar x$ such that $y\in X$, $\phi(y)\leqq\phi(\bar x)$, $\langle\nabla\phi (\bar x), y-\bar x\rangle=0$. We say that $\phi$ is  {\em strictly $2$-pseudoconvex} at  $\bar x$ if, for all $y\in X$, $y\neq \bar x$, the following implications hold:
		\begin{align*}
		&\phi (y)\leqq\phi (\bar x)\ \ \text{implies}\ \ \langle\nabla\phi (\bar x), y-\bar x\rangle\leqq 0;
		\\
		&\phi (y)\leqq \phi (\bar x)\ \ \text{and}\ \ \langle\nabla\phi (\bar x), y-\bar x\rangle=0 \ \ \text{imply}\ \ \phi^{\prime\prime}(\bar x, y-\bar x)<0.
		\end{align*}		
		}	
\end{definition}

It follows from this definition that every strictly $2$-pseudoconvex function is $2$-pseudoconvex. The converse does not hold. For example, the function $f_1$   in Example \ref{linear-example} is $2$-pseudoconvex at $\bar x=(0,0)$ but not strictly $2$-pseudoconvex. Indeed, for $y=(0, 1)$, we have $f_1(y)=f_1(\bar x)$, $\langle\nabla f_1(\bar x), y\rangle=0$, and $f_1^{\prime\prime}(\bar x; y-\bar x)=0$. Thus, $f_1$ is not strictly $2$-pseudoconvex at $\bar x$. We also see that $\bar x$ is not a strict global efficient solution of \eqref{linear-problem-exam}. Therefore the sufficient conditions of  Theorem  \ref{sufficient conditions-II} do not guarantee for a strict global  efficient solution even for linear vector optimization problems.  A natural question arises: {\em How does one obtain sufficient optimality conditions for strict global  efficient solutions of \eqref{problem}?} The rest of this section is aimed at solving the problem.

The following result gives sufficient conditions of Karush--Kuhn--Tucker type  for a strict global  efficient solution of \eqref{problem} under the assumptions that the objective functions are strictly $2$-pseudoconvex and the constraint functions are quasiconvex.

\begin{theorem}\label{sufficient conditions-III} Let $\bar x$ be a feasible point of \eqref{problem}. Suppose that $f_j$, $j\in J$, $g_i$, $i\in I(\bar x)$ are  second-order directionally differentiable at $\bar x$ in every critical direction $d\in \mathcal{C}(\bar x)$, $f_j$, $j\in J$, are strictly $2$-pseudoconvex at $\bar x$, $g_i$, $i\in I(\bar x)$ are quasiconvex at $\bar x$. If for each  $d\in\mathcal{C}(\bar x)$, there exist $\mu\in\mathbb{R}_+^p\setminus\{0\}$ and $\lambda\in\mathbb{R}_+^m$  satisfying conditions \eqref{first-order-condition}--\eqref{complementary-condition-II},	then $\bar x$ is a strict global  efficient solution of \eqref{problem}.
\end{theorem}
\begin{proof} The proof is quiet similar to that of the proof of    Theorem \ref{sufficient conditions-II}, so omitted. $\hfill\Box$
\end{proof}

The next result gives sufficient conditions of Fritz-John type  for a strict global  efficient solution of \eqref{problem} with strictly $2$-pseudoconvex data and extends \cite[Theorem 4]{Ginchev08} to the vector case.
\begin{theorem} Let $\bar x$ be a feasible point of \eqref{problem}. Suppose that $f_j$, $j\in J$, $g_i$, $i\in I(\bar x)$ are  second-order directionally differentiable at $\bar x$ in every critical direction $d\in \mathcal{C}(\bar x)$, $f_j$, $j\in J$, $g_i$, $i\in I(\bar x)$, are strictly $2$-pseudoconvex at $\bar x$. If for each $d\in \mathcal{C}(\bar x)$, there exists  $(\mu, \lambda)\in (\mathbb{R}^p_+\times \mathbb{R}^m_+)\setminus\{(0,0)\}$ satisfying conditions \eqref{first-order-condition}--\eqref{complementary-condition-II},	then $\bar x$ is a strict global  efficient solution of \eqref{problem}.
\end{theorem}
\begin{proof}  Arguing by contradiction, suppose that there exists $x\in \mathcal{F}$ such that $x\neq \bar x$ and $f(x)\leqq f(\bar x)$. An analysis similar to the one made in the proof of  Theorem \ref{sufficient conditions-II} shows that $x-\bar x\in \mathcal{C}(\bar x)$. Let $(\mu, \lambda)\in \mathbb{R}^p_+\times \mathbb{R}^m_+$ be a nonzero Lagrange multiplier  satisfying conditions \eqref{first-order-condition}--\eqref{complementary-condition-II}. Then we have
\begin{equation*}
\begin{cases}
&\langle\nabla f_j(\bar x), x-\bar x\rangle=0,\ \ j\in \text{supp}\,\mu,
\\
&\langle\nabla g_i(\bar x), x-\bar x\rangle=0,\ \ i\in \text{supp}\,\lambda,
\end{cases}
\end{equation*}
where $\text{supp}\,\mu$ and $\text{supp}\,\lambda$ are defined as in \eqref{supp}.
By the strictly $2$-pseudoconvexity of $f_j$, $j\in J$, $g_i$, $i\in I(\bar x)$, at $\bar x$, we have
\begin{equation*}
\begin{cases}
&f_j^{\prime\prime}(\bar x; x-\bar x)<0,\ \ j\in \text{supp}\,\mu,
\\
&g_i^{\prime\prime}(\bar x; x-\bar x)<0,\ \ i\in \text{supp}\,\lambda.
\end{cases}
\end{equation*}
Since $(\mu, \lambda)\neq 0$, it follows that
$$\text{supp}\,\mu\cup \text{supp}\,\lambda\neq \emptyset.$$
Thus,
\begin{align*}
\sum_{j=1}^p \mu_j f_j^{\prime\prime}(\bar x; x-\bar x)+\sum_{i=1}^m \lambda_i g_i^{\prime\prime}(\bar x;  x-\bar x)=&\sum_{j\in \text{supp}\,\mu} \mu_j f_j^{\prime\prime}(\bar x; x-\bar x)
\\
&+\sum_{i\in \text{supp}\,\lambda} \lambda_i g_i^{\prime\prime}(\bar x;  x-\bar x)<0,
\end{align*} 
contrary to \eqref{secon-order-condition}. $\hfill\Box$
\end{proof}

We now introduce sufficient conditions of Karush--Kuhn--Tucker type for a strict global  efficient solution of \eqref{problem} with quasiconvex data.
\begin{theorem}\label{sufficient conditions-V} Let $\bar x$ be a feasible point of \eqref{problem} and the functions $f_j$, $j\in J$, $g_i$, $i\in I(\bar x)$ be quasiconvex  at $\bar x$. Suppose that $f_j$, $j\in J$, $g_i$, $i\in I(\bar x)$ are second-order directionally differentiable at $\bar x$ in every critical direction $d\in \mathcal{C}(\bar x)$. If for each $d\in \mathcal{C}(\bar x)\setminus\{0\}$, there exist $\mu\in \mathbb{R}^p_+\setminus\{0\}$ and $\lambda\in  \mathbb{R}^m_+$  such that
	\begin{align}
	&\sum_{j=1}^p \mu_j \nabla f_j(\bar x)+\sum_{i=1}^m \lambda_i\nabla g_i(\bar x)=0, \label{first-order-condition-III}
	\\
	&\sum_{j=1}^p \mu_j f_j^{\prime\prime}(\bar x; d)+\sum_{i=1}^m \lambda_i g_i^{\prime\prime}(\bar x; d)> 0, \label{secon-order-condition-III}
	\\
	&\lambda_i g_i(\bar{x})=0,\ \ \ i\in I,\label{complementary-condition-III}
	\end{align}	
	then $\bar x$ is a strict global efficient solution of \eqref{problem}.
\end{theorem}
\begin{proof}  The proof is indirect. Suppose that $\bar x$ is not a  strict global  efficient solution of \eqref{problem}. Then, there exists $x\in \mathcal{F}$ such that $x\neq \bar x$ and $f(x)\leqq f(\bar x).$ This implies that
\begin{equation*}
\begin{cases}
f_j(x)\leqq f_j(\bar x), \ \ \ \forall j\in J,
\\
g_i(x)\leqq g_i(\bar x), \ \ \ \forall i\in I(\bar x).
\end{cases}
\end{equation*}
By Lemma  \ref{lemma-quasiconvex} and the quasiconvexity of  $f_j$  and $g_i$ at $\bar x$, we have
 \begin{equation*}
 \begin{cases}
 \langle \nabla f_j(\bar x), x-\bar x\rangle\leqq 0, \ \ \ \forall j\in J,
 \\
 \langle \nabla g_i(\bar x), x-\bar x\rangle\leqq 0, \ \ \ \forall i\in I(\bar x).
 \end{cases}
 \end{equation*}
Put $d=x-\bar x$. Then, $d$ is a nonzero critical direction at $\bar x$. Using the assumptions of the theorem we deduce that there exist $\mu\in\mathbb{R}^p_+\setminus\{0\}$ and $\lambda\in  \mathbb{R}^m_+$ satisfying conditions \eqref{first-order-condition-III}--\eqref{complementary-condition-III}. For each $j\in J$, again by the quasiconvexity of $f_j$, we have
 $$f_j(\bar x+td)\leqq f_j(\bar x), \ \ \forall t\in [0, 1].$$
 By Lemma \ref{Taylor-expansion}, for all  $t>0$ small enough, one has
$$0\geq f_j(\bar x+td)-f_j(\bar x)=t\langle\nabla f_j(\bar x), d\rangle+\frac{1}{2}t^2f_j^{\prime\prime}(\bar x; d)+o(t^2), \ \ \forall j\in J.$$
Consequently,
\begin{equation}\label{equ:27}
\langle\nabla f_j(\bar x), d\rangle+\frac{1}{2}tf_j^{\prime\prime}(\bar x; d)+o(t)\leqq 0
\end{equation}
for all $t>0$ small enough and $j\in J$.

Similarly, for each $i\in I(\bar x)$ and $t>0$ small enough, we have
 \begin{equation}\label{equ:28}
 \langle\nabla g_i(\bar x), d\rangle+\frac{1}{2}tg_i^{\prime\prime}(\bar x; d)+o(t)\leqq 0.
 \end{equation}
 Now multiplying \eqref{equ:27} by $\mu_j$ and \eqref{equ:28} by $\lambda_i$ and then adding, we get
 \begin{align*}
 0\geqq \bigg\langle\sum_{j\in J} \mu_j\nabla f_j(\bar x)&+\sum_{i\in I(\bar x)}\lambda_i\nabla g_i(\bar x), d\bigg\rangle 
 \\
 &+\frac{1}{2}t\left(\sum_{j\in J} \mu_j f_j^{\prime\prime}(\bar x; d)+\sum_{i\in I(\bar x)}\lambda_i g_i^{\prime\prime}(\bar x; d)\right)+o(t).
 \end{align*}
 From this and \eqref{first-order-condition-III} it follows that
 \begin{equation}\label{equ:29}
 \sum_{j\in J} \mu_j f_j^{\prime\prime}(\bar x; d)+\sum_{i\in I(\bar x)}\lambda_i g_i^{\prime\prime}(\bar x; d)+o(1)\leqq 0
 \end{equation}
 for all $t>0$ small enough. Letting $t\downarrow 0$ in \eqref{equ:29}, we obtain
 \begin{equation*}
 \sum_{j\in J} \mu_j f_j^{\prime\prime}(\bar x; d)+\sum_{i\in I(\bar x)}\lambda_i g_i^{\prime\prime}(\bar x; d)\leqq 0,
 \end{equation*}
 contrary to \eqref{secon-order-condition-III}. $\hfill\Box$
\end{proof}

By replacing the quantity $(y-\bar x)$ in \eqref{quasiconvex-character} by a function $\eta(y, \bar x)$, Hanson \cite{Hanson81} introduced a new concept of quasiinvex functions as a generalization of quasiconvex functions as follows.
\begin{definition}[{see \cite{Hanson81}}]
	{\rm  Suppose that the function $\phi\colon X\to \mathbb{R}$ is differentiable at $\bar x\in X$. We say that $\phi$ is {\em quasiinvex} at $\bar x\in X$ with respect to $\eta(\,\cdot\,, \bar x)\colon X\to \mathbb{R}$ if the following condition holds:
		\begin{equation*}
		\left(y\in X, \phi (y)\leqq \phi (\bar x)\right) \Longrightarrow	\langle\nabla\phi (\bar x), \eta(y,\bar x)\rangle\leqq 0.
		\end{equation*}
	}
\end{definition}

\begin{remark}{\rm We have the following observations:
\begin{itemize}
  \item We note here that  the concepts of  quasiinvex functions  and  quasiconvex functions can be very different. For example, let $\phi (x)=x^3$ for all $x\in\mathbb{R}$ and $\bar x=0$. Since $\nabla\phi(\bar x)=0$, $\phi$ is quasiinvex at $\bar x$ with respect to any function $\eta(\,\cdot\,, \bar x)$. Moreover, it is easy to check that $\phi$ is quasiconvex at $\bar x$. Thus, if $\phi(y)\leqq \phi(\bar x)$, then
$$\phi (\bar x+  t(y-\bar x))\leqq \phi (\bar x), \ \ \ \forall t\in [0, 1].$$
This property does not hold for  quasiinvex functions. Indeed, let $\eta(y, \bar x)=-y-\bar x$ for all $y\in\mathbb{R}$. Then, $\phi$  is quasiinvex at $\bar x$ with respect to $\eta(\,\cdot\,, \bar x)$. However, for $y=-1$, we see that $\phi (y)<\phi (\bar x)$ and
$$\phi (\bar x+  t\eta(y, \bar x))=t^3>  \phi (\bar x),\ \ \ \forall t>0.$$	
  \item The following example indicates that if the quasiconvexity of the objective functions and the active constraint functions is replaced by the quasiinvexity of these functions, then Theorem \ref{sufficient conditions-V} may not be valid. This shows that Theorem 5 in \cite{Santos2013} is not correct.
 \end{itemize}
		}	
\end{remark}

\begin{example}\label{counter-example}{\rm  Consider the following problem:
\begin{align*}
&\text{min}\,_{\mathbb{R}_+}\, f(x)
\\
&\text{s. t.} \ \ x\in \mathcal{F}:=\{x\in \mathbb{R}\,:\, g(x)\leqq 0\},
\end{align*}	
where  $f, g\colon\to\mathbb{R}$ are two functions defined by
$$f(x):=-x^3, g(x):=-x^3+x^2, \ \ \forall x\in\mathbb{R}.$$
Obviouly $\bar x:=0\in \mathcal{F}$. Since $\nabla f(\bar x)=0$ and $\nabla g(\bar x)=0$, we have that $f$ and $g$ are quasiinvex at $\bar x$ with respect to any function $\eta(\,\cdot\,, \bar x)$. However, the function $g$ is not quasiconvex at $\bar x$. Indeed, for $x=1$, we have $g(x)=g(\bar x)$ and
$$g(\bar x+t(x-\bar x))=t^2(1-t)>g(\bar x),\ \ \ \forall t\in \left(0, 1\right)$$
as required.

Clearly, $\mathcal{C}(\bar x)=\mathbb{R}$. We can choose the same Lagrange multipliers  $\mu\in\mathbb{R}_+\setminus\{0\}$ and $\lambda\in  \mathbb{R}_+$ satisfying conditions \eqref{first-order-condition-III}--\eqref{complementary-condition-III} for all critical directions $d\in \mathcal{C}(\bar x)\setminus\{0\}$; for example, $(\mu, \lambda)=(1,1)$. However, since $x=1\in \mathcal{F}$ and $f(1)<f(\bar x)$,  $\bar x$ is not a global minimum solution of $f$ on $\mathcal{F}$. This shows that \cite[Theorem 5]{Santos2013} is not correct even for scalar optimization problems with $C^2$ data.
}
\end{example}


\section*{Acknowledgments} {J.-C. Yao and C.-F. Wen are supported by the Taiwan MOST [grant number 107-2923-E-039-001-MY3], [grant number 107-2115-M-037-001],  respectively, as well as the grant from Research Center for Nonlinear 	Analysis and Optimization, Kaohsiung Medical University, Taiwan. Y.-B. Xiao is supported by the National Natural Science Foundation of China (11771067).}


\begin{thebibliography}{99}
\bibitem{Ben-Tal80}
{ Ben-Tal, A.:} {Second-order and related extremality conditions in nonlinear programming}. J. Optim. Theory Appl. {\bf 31}, 143--165  (1980)

\bibitem{Bel-Tal82}
{ Ben-Tal, A.,  Zowe, J.:} {Necessary and sufficient optimality conditions for a class of nonsmooth minimization problems.} Math. Program.  {\bf24}, 70--92  (1982)

\bibitem{Bel-Tal85}
{ Ben-Tal, A.,  Zowe, J.:} {Directional derivatives in nonsmooth optimization.} J. Optim. Theory Appl.  {\bf 47}, 483--490  (1985)


\bibitem{Bonnans2000}
{Bonnans, J.F., Shapiro, A.:} {Perturbation Analysis of Optimization Problems.} Springer, New York (2000)

\bibitem{Cominetti}
{Cominetti, R., Correa, R.:} {A generalized second-order derivative in nonsmooth optimization}. SIAM J. Control Optim. {\bf 28}, 789--809  (1990)
	
\bibitem{Demyanov1974}
{Demyanov, V.F., Pevnyi, A.B.:} {Expansion with respect to a parameter of the extremal values of game problems.} U.S.S.R. Comput. Math. and Math. Phys.  {\bf 14}, 33--45  (1974)

\bibitem{Ehrgott2006}
{ Ehrgott, M.:} {Multicriteria Optimization.} Springer Science and Business Media (2006)


\bibitem{Ginchev05}
{ Ginchev, I.,  Guerraggio,  A.,   Rocca, M.:} {Second-order conditions in $C^{1,1}$ constrained vector optimization.} Math. Program. {\bf 104}, 389--405  (2005)

\bibitem{Ginchev06}
{ Ginchev, I.,  Guerraggio,  A.,   Rocca, M.:} {From scalar to vector optimization.} Appl. Math. {\bf 51}, 5--36  (2006)

\bibitem{Ginchev07}
{Ginchev, I.,  Ivanov, V.I.:} {Higher-order pseudoconvex functions}. In: Konnov, I., Luc, D.T., Rubinov, A.M. (eds.) Proceedings of the 8th International Symposium on Generalized Convexity and Monotonicity. Lecture Notes in Economics and Mathematical Systems, vol. 583, pp. 247--264. Springer, Berlin (2007)


\bibitem{Ginchev08}
{Ginchev, I.,  Ivanov, V.I.:} {Second-order optimality conditions for problems with $C^1$ data.} J. Math. Anal. Appl. {\bf 340}, 646--657  (2008)

\bibitem{Ginchev10}
{ Ginchev, I.,  Guerraggio, A.,  Rocca, M.:}  {Second-order Dini set-valued directional derivative in $C^{1,1}$ vector optimization}. Optim. Methods Softw. {\bf 25},  75--87  (2010)

\bibitem{Ginchev11}
{ Ginchev, I.,  Guerraggio, A.:} {Second-order conditions for constrained vector optimization problems with $l$-stable data}. Optimization  {\bf 60}, 179--199  (2011)

\bibitem{Hanson81}
{ Hanson, M.A.:} {On sufficiency of Kuhn--Tucker conditions.} {J. Math. Anal. Appl. }  {\bf 80}, 545--550  (1981)

\bibitem{Huy16}
{ Huy, N.Q.,  Tuyen, N.V.:} { New second-order optimality conditions for a class of differentiable optimization problems}. J. Optim. Theory Appl. {\bf 171}, 27--44  (2016)

\bibitem{Huy17}
{ Huy, N.Q.,  Kim, D.S.,  Tuyen, N.V.:} {New second-order Karush--Kuhn--Tucker optimality conditions for vector optimization.}  Appl. Math. Optim.  (2017). https://doi.org/10.1007/s00245-017-9432-2




\bibitem{Jimenez02} { Jim\'enez, B.:} {Strict efficiency in vector optimization.}  J. Math. Anal. Appl. {\bf 265}, 264--284  (2002)

\bibitem{Jimenez03} { Jim\'enez, B.,  Novo, V.:} {First and second order sufficient conditions for strict minimality in nonsmooth vector optimization.}  J. Math. Anal. Appl.  {\bf 284}, 496--510  (2003)


\bibitem{Khanh07}
{ Khanh,P.Q.,  Tuan, N.D.:} {Optimality conditions for nonsmooth multiobjective optimization using Hadamard directional derivatives.}  J. Optim. Theory Appl.  {\bf 133}, 341--357 (2007)

\bibitem{Khanh08}
{ Khanh,P.Q.,  Tuan, N.D.:} {First and second-order approximations as derivatives of mappings in optimality conditions for nonsmooth vector optimization}. Appl. Math. Optim.  {\bf 58},  147--166  (2008)

\bibitem{Luu17} {  Luu, D.V.:}  {Second-order necessary efficiency conditions for nonsmooth vector equilibrium problems.}  J. Global Optim.  {\bf 70},  437--453  (2018)

 

\bibitem{Mangasarian69} { Mangasarian, O.L.:} {Nonlinear Programming.} McGraw Hill, New York (1969)

\bibitem{Mifflin04}
{ Mifflin, R.,   Sagastiz\'abal, C.:} {On the relation between $\mathscr{U}$-Hessians and second-order epi-derivatives.}  European J. Oper. Res.  {\bf 157},  28--38  (2004)


\bibitem{Santos2013}
{ Santos, L.B.,   Osuna-G\'omez, R.,   Hern\'andez-Jim\'enez, B.,   Rojas-Medar, M.A.:} {Necessary and sufficient second order optimality conditions for multiobjective problems with $C^1$ data.} Nonlinear Anal.  {\bf 85}, 192--203  (2013)

\bibitem{Tuy64} { Tuy, H.:} {Sur les in\'egalit\'es lin\'eaires. } Colloq. Math.  {\bf 13}, 107--123  (1964)

\bibitem{Tuyen18}
{Tuyen, N.V., Huy, N.Q., Kim, D.S.:} {Strong second-order Karush--Kuhn--Tucker optimality conditions for vector optimization.}  Appl. Anal.  (2018). https://doi.org/10.1080/00036811.2018.1489956

\bibitem{Xiao18}
{Xiao, Y.B.,   Tuyen, N.V.,  Yao, J.C.,  Wen, C.F.:} {Locally Lipschitz vector optimization problems: Second-order constraint qualifications, regularity condition, and KKT necessary optimality conditions.} (2018). https://arxiv.org/abs/1710.03989
\end{thebibliography}
\end{document}